\documentclass[12pt,leqno,twoside]{amsart}

\usepackage[latin1]{inputenc}
\usepackage[T1]{fontenc}
\usepackage[colorlinks=true, pdfstartview=FitV, linkcolor=blue, citecolor=blue, urlcolor=blue]{hyperref}
\usepackage{amstext,amsmath,amscd, bezier,indentfirst,amsthm,amsgen,enumerate, geometry}
\usepackage[all,knot,arc,import,poly]{xy}
\usepackage{amsfonts,color, soul}  
\usepackage{amssymb}
\usepackage{latexsym}
\usepackage{epsfig}
\usepackage{graphicx}
\usepackage{srcltx}
\usepackage{enumitem}

\newtheorem{theorem}{Theorem}[section]
\newtheorem{corollary}[theorem]{Corollary}
\newtheorem{lemma}[theorem]{Lemma}
\newtheorem{conjecture}[theorem]{Conjecture}
\newtheorem{econjecture}[theorem]{Equivalence Conjecture}
\newtheorem{problem}[theorem]{Problem}
\newtheorem{proposition}[theorem]{Proposition}
\theoremstyle{definition}
\newtheorem{definition}[theorem]{Definition}
\theoremstyle{remark}
\newtheorem{remark}[theorem]{\sc Remark}
\newtheorem{example}[theorem]{\sc Example}

\renewcommand{\Box}{\square}    

\newcommand{\proj}{{\rm{proj}}}

\newcommand{\Sing}{{\rm{Sing\hspace{2pt}}}}

\newcommand{\Disc}{{\rm{Disc\hspace{2pt}}}}

\renewcommand{\i}{{\rm{int}}}

\newcommand{\im}{{\rm{Im\hspace{2pt}}}}

\newcommand{\e}{\varepsilon}
\newcommand{\m}{\setminus}

\newcommand{\fin}{\hspace*{\fill}$\Box$\vspace*{2mm}}

\newcommand{\bR}{{\mathbb R}}
\newcommand{\bC}{{\mathbb C}}

\begin{document}

\title[Milnor-Hamm sphere fibrations]{Milnor-Hamm sphere fibrations and the equivalence problem}

\author{\sc Raimundo N. Ara\'ujo dos Santos}
\address{ICMC,
Universidade de S\~ao Paulo,  Av. Trabalhador S\~ao-Carlense, 400 -
CP Box 668, 13560-970 S\~ao Carlos, S\~ao Paulo,  Brazil}
\email{rnonato@icmc.usp.br, raimundo.icmc@gmail.com}

\author{\sc Maico F. Ribeiro}
\address{Universidade Federal do Espirito Santo,  Av. Fernando Ferrari, 514 - Goiabeiras - CP 29.075-910 Vit\'oria, Espirito Santo, Brazil}
\email{maico.ribeiro@ufes.br}

\author{Mihai Tib\u ar}
\address{Univ. Lille, CNRS, UMR 8524 -- Laboratoire Paul Painlev\'e, F-59000 Lille,
France}  
\email{mihai-marius.tibar@univ-lille.fr}



\thanks{The authors acknowledge the support of USP-COFECUB Uc Ma 163-17.  MT acknowledges
the support of the Labex CEMPI
(ANR-11-LABX-0007-01). RNAS acknowledges the Fapesp grant 2017/20455-3 and the CNPq grant 313780/2017-0.}

\begin{abstract}
 
We introduce the sphere fibration for real map germs  with radial discriminant
and we address the problem of its equivalence with the Milnor-Hamm tube fibration. Under natural 
conditions, we prove the existence  of open book structures with singularities and  solve the equivalence problem.

\end{abstract}

\maketitle
\section{Introduction}

Let $G:(\bR^{m},0) \rightarrow (\bR^{p}, 0)$, $p\ge 2$, be a non-constant real analytic map germ. Under the condition that $G$ has isolated singularity at $0\in \bR^{m}$, it  was shown by Milnor \cite{Mi} that there exists a tube fibration and a sphere fibration. Together they contribute to the definition of a \emph{higher open book structure} as explained in  \cite{dST0, dST1, ACT} in the more general case when the singular set $\Sing G$ is non-isolated but still included in the central fibre $G^{-1}(0)$. Milnor \cite{Mi} construction of  a sphere fibration by the method of blowing away the tube fibration holds under certain conditions (see the discussion in \cite[\S 2]{dST1}), providing a topological equivalence between the empty tube  fibration and the sphere fibration.

  In case of holomorphic functions $f:(\bC^{n},0) \to (\bC, 0)$, Milnor  \cite{Mi} proved  that the sphere fibration  is induced by the special map $f/\| f\|$.  However this particular construction does not extend to real map germs, as already suggested by Milnor in  \cite[\S 11]{Mi}.  There have been 
several successful tries to add up supplementary conditions so that  the map $G/\| G\|$ defines a fibration, e.g. \cite{Ja1, RSV, RA, dST0, dST1, CSS1,  A1} etc. 
In case this fibration exists, there remains the question if it is equivalent  to the empty tube fibration. This  seems not to have been solved even in the case when $G$ has isolated singularity\footnote{The proof claimed in \cite{CSS1} seems to be incomplete as pointed out in \cite{Ha} and \cite{Ri}.}, the setting in which Milnor \cite[\S 9]{Mi} had  already given several related results. We formulate this problem here as the \emph{Equivalence Conjecture \ref{conjec}}.

\medskip

Recently one started to enrich this landscape by treating the case of  a positive dimensional discriminant $\Disc G$.  As one can easily see, the classical case $\Disc G = \{0\}$ remains a very special situation,  which for instance never happens in case of  maps $(\bC^{n}, 0) \to (\bC^{p}, 0)$ with $p>1$ defining \emph{isolated complete intersection singularities}.

The tube fibration for positive dimensional discriminant, predicted in \cite{ACT}, has been introduced recently in \cite{ART} under the name \emph{Milnor-Hamm tube fibration}, and several  new classes of singular map germs with such fibration  have been presented.    This means that over each connected component of the complement of $\Disc G$ there is a well-defined locally trivial fibration, and there  are finitely many such components\footnote{A subanalytic set has locally finitely many connected components, see e.g. \cite{BM}.}.

\medskip

According to Milnor's program \cite{Mi} detailed in \cite{dST1}, there are two more steps in order to define  \emph{open book structures with singularities}.   Our paper is devoted to this task in the most reasonable setting of a \emph{radial discriminant}. We introduce the \emph{Milnor-Hamm sphere fibration}, we give natural sufficient conditions such that this exists, and we exhibit several such classes of singular maps.

We then state the problem of the equivalence with the corresponding Milnor-Hamm empty tube fibration and
we show how to solve it in our general setting under natural supplementary conditions.

 Several conditions for the existence of the Milnor vector field are presented in \cite{AR}, and further details can be found in \cite{Ri}.
  

\section{The singular tube fibration}
\subsection{Nice map germs}\label{ss:tame}

Given a non-constant analytic map germ $G:(\bR^{m},0) \rightarrow (\bR^{p}, 0)$, $m\ge p >0$, the set germ $\Sing G$ is well defined on the source space but the images $G(\Sing G)$ and $\im G$ are in general not well-defined as germs of subanalytic sets, see \cite{ART, JT}.  If they are, then we say that  $G$ is  a \emph{nice map germ}.

A simple example of a non-nice map germ is $(x, xy) : (\bC^{2},0)\to (\bC^{2}, 0)$, for which $\im G$ is not a germ. Under these notations, one has the following results about the existence of nice map germs\footnote{For more refined results concerning \emph{nice map germs} $(f,g)$ and $f\bar g$ in $n$ variables, one may consult the recent works \cite{JT, JT2}.}: 

\noindent \cite[Lemma 2.4]{ART}
\emph{If $\Sing G \cap G^{-1}(0) \subsetneq  G^{-1}(0)$ then $\im G$ contains an open neighbourhood of the origin.
}

\noindent \cite[Theorem 2.7]{ART}
\emph{Let $f, g: (\bC^{n}, 0)\to (\bC, 0)$ be complex polynomials with no common factor of order $\ge 1$. Then $f\bar g : (\bC^{n}, 0)\to (\bC, 0)$ is a nice map germ.}

As introduced in  \cite{ART},  we shall call \emph{discriminant} of a nice map germ $G$ the following set:
\begin{equation}\label{eq:disc}
\Disc G := \overline{G(\Sing G)} \cup \partial \overline{\im G}
\end{equation}
where the boundary $\partial \overline{\im G} := \overline{\im G} \m \i(\im G)$  is a closed subanalytic  proper subset of $\bR^{p}$ and well-defined as a set germ,  where $\i A := \mathring{A}$ denotes the $p$-dimensional interior of a semianalytic set $A\subset \bR^{p}$ (hence it is empty whenever $\dim A <p$), and $\overline{A}$ denotes the closure of it.
 
It follows from the definition that $\Disc G$ is a closed subanalytic set of dimension strictly less than $p$, well-defined as a germ.

\subsection{The Milnor-Hamm fibration}
\begin{definition}\cite[Definition 2.1]{ART}\label{d:tube}
Let $G:(\bR^{m},0) \rightarrow (\bR^{p}, 0)$ be a non-constant analytic  nice map germ.
We say that $G$ has a {\em  Milnor-Hamm tube fibration} if  for any $\e > 0$ small enough, there exists  $0<\eta \ll \e$ such that the restriction:
\begin{equation}\label{eq:tube}
 G_| :  B^{m}_{\e} \cap G^{-1}( B^{p}_\eta \m \Disc G) \to  B^{p}_\eta \m \Disc G
\end{equation}
 is a locally trivial smooth fibration which is independent, up to diffeomorphisms, of the choices of small enough $\e$ and $\eta$. 
 
 We then also say that the restriction of \eqref{eq:tube} over a small enough sphere (still denoted by $S^{p}_\eta$ but keeping in mind that the radius is slightly smaller than the $\eta$ in \eqref{eq:tube}): 
 \begin{equation}\label{eq:emptytube}
 G_| :  B^{m}_{\e} \cap G^{-1}( S^{p}_\eta \m \Disc G) \to  S^{p}_\eta \m \Disc G
\end{equation}
 is a {\em  Milnor-Hamm empty tube fibration}.
 \end{definition}
 
One defines in \cite{ART} a more general notion of stratified tube fibration called \emph{singular Milnor tube fibration} by considering in addition
 all singular fibres over the stratified discriminant. In all cases, the tube fibration is a collection  of  finitely many 
 fibrations over path-connected subanalytic sets.


\subsection{$\rho$-regularity of map germs.}
Let $U \subset \bR^m$ be an open set,  $0\in U$,  and let $\rho:U \to \bR_{\ge 0}$ be the square of the Euclidean distance to $0$. 
We recall the following definition from \cite{ART}:	

\begin{definition}\label{d:ms}
	Let 	$G:(\bR^m, 0) \to (\bR^p,0)$ be a non-constant analytic nice map germ. The set germ at the origin:
	\[M(G):=\overline{\left\lbrace x \in U \mid \rho \not\pitchfork_x G \right\rbrace } \]
	is called the set of \textit{$\rho$-nonregular points} of $G$, or the \emph{Milnor set of $G$}. 
\end{definition}
 
The following inclusion of set germs at the origin will play an important role, where $V_G := G^{-1}(0)$:
\begin{equation}\label{eq:main}
\overline{M(G)\m G^{-1}(\Disc (G))}\cap V_G \subseteq \{0\}.
\end{equation}
Condition  \eqref{eq:main} is a direct extension of the condition used in  \cite{dST0, dST1, Ma, ACT} in case $\Disc (G) =  \{0\}$; it was shown that it is implied by the Thom regularity condition,  in \emph{loc.cit.} and several other papers. The  reciprocal is however  not true,  counterexamples are provided in \cite{Obw, ACT}, see also \cite{Oka4, PT}.

Condition  \eqref{eq:main} enables the  following existence result proved in \cite[Lemma 3.3]{ART}: 

\emph{Let $G:(\bR^m, 0) \to (\bR^p,0)$ be a non-constant analytic nice map germ. If $G$ satisfies condition \eqref{eq:main}, then $G$ has a Milnor-Hamm tube fibration \eqref{eq:tube}.}



\section{The Milnor-Hamm sphere fibration} \label{S-emhsph} 
  
  We introduce a natural condition  under which one may define sphere fibrations whenever $G$ is  a nice map germ  and  $\Disc G$ is positive dimensional.
 \begin{definition}\label{d:radialdisc}
 	Let $G:(\bR^{m},0)\to (\bR^{p},0)$ be a real analytic nice map germ. We say that its discriminant $\Disc G$ is \emph{radial} if it is the germ of a real cone with vertex at the origin (i.e. a union of real half-lines at the origin), or just the origin. 
 \end{definition} 
 
 \begin{example}\label{ex3}
 	Let $f,g:(\mathbb{C}^{n},0)\to (\mathbb{C},0)$ be holomorphic function germs such that $f$ and $g$ do not have any common factor of order $>0$.   Then $f\bar g : (\bC^n,0)\to (\bC, 0)$ is a nice map germ by  \cite[Theorem 2.7]{ART}. If $\Disc (f,g)$ is radial, then obviously $\Disc f\bar g$  is radial, e.g. by \cite[Lemma 2.4]{PT}.
 \end{example}

 \begin{example}\label{ex4}
Let $f:(\bR^m,0) \to (\bR^{p},0)$ be a real analytic nice map germ, and let $g:(\bR,0) \to (\bR,0)$ be an analytic invertible germ, such that $f$ and $g$ are in separate variables.  Then the map germ $G:=(f,g):(\bR^{m}\times \bR,0) \to (\bR^{p}\times \bR,0)$  has $\Sing G = \Sing f \times \bR$ and it is \emph{nice}. If moreover $\Disc f$ is radial, then $\Disc G$ is radial. 
 \end{example}

Let  $G:U\to \bR^{p}$ be a representative of the nice map germ $G$ in some open set $U \ni 0$.  We consider the map
\begin{equation} \label{eq:spherefib}
 {\Psi_G} := \frac{G}{\| G\|}: U \m V_G \to  S_{1}^{p-1}.
 \end{equation}
 

If $G$ is a nice map germ such that $\Disc G$ is \emph{radial}, then it follows from the definitions  that the restriction:  
 	\begin{equation}\label{ssf}    
 	{\Psi_G}_{|} : S_{\e}^{m-1}\m G^{-1}(\Disc G )\to S_{1}^{p-1}\m \Disc G   
 	\end{equation}
is well defined for any $\e >0$ small enough.

\begin{definition} \label{d:spherefib}
We say that  the nice map germ $\displaystyle{G:(\bR^{m},0)\to (\bR^{p},0)}$  with radial discriminant has a 
\textit{Milnor-Hamm sphere fibration}
if the restriction map \eqref{ssf} is a locally trivial smooth fibration which is independent, up to diffeomorphisms, of the choice of $\e>0$ provided it is small enough.

Let $M({\Psi_G})$ be the Milnor set of the map \eqref{eq:spherefib}, i.e. the germ at the origin of the $\rho$-nonregular points of $\Psi_G$, cf Definition \ref{d:ms}.  We  say that  \textit{$\Psi_G$ is $\rho$-regular} if:  
 	\begin{equation}\label{eq:reg}
 	M ({\Psi_G}) \subset G^{-1}(\Disc G).
 	\end{equation}
 \end{definition}
 
 \begin{remark}\label{r:welldef}
 The reason for restricting our definition of Milnor-Hamm sphere fibration to radial discriminants  is precisely that 
 the restriction \eqref{ssf} of the map $\frac{G}{\| G\|}$ is well-defined if and only if  the discriminant is radial.  For non-radial discriminants there is a different way to define singular open books without using the map $\frac{G}{\| G\|}$, see Remark \ref{rem:general} and Conjecture \ref{c:conj}.
\end{remark}

 \subsection{Existence of Milnor-Hamm sphere fibrations.}
The following existence criterion extends the case  $\Disc G =\{ 0\}$ considered in \cite[Theorem 1.3]{ACT}.

 \begin{theorem}\label{sf2} 
  Let $\displaystyle{G:(\bR^{m},0)\to (\bR^{p},0)}$, $m>p\ge2$, be a non-constant  analytic nice map germ with radial discriminant,  satisfying the condition \eqref{eq:main}. 
    If ${\Psi_G}$ is $\rho$-regular then $G$ has a  Milnor-Hamm sphere fibration.

 \end{theorem}

 \begin{proof}
The condition \eqref{eq:reg} controls the topology of the map ${\Psi_G}$ on the complementary of a tubular neighbourhood of 
$V_{G}$, while  the condition \eqref{eq:main}
controls the behaviour of the map $\Psi_{G}$ close to $V_{G}$.  Both conditions are essential, as one can see in many examples.

\medskip
\noindent
	\emph{Step 1}. Under the condition \eqref{eq:main},  by \cite[Lemma 3.3]{ART}, the restriction
	\begin{equation}\label{eq16}
	G_{|}: S_{\e}^{m-1}\cap G^{-1}(\overline{B}^{p}_{\eta} \m \Disc G) \to \overline{B}^{p}_{\eta} \m \Disc G
	\end{equation}	
	is a locally trivial  fibration for any small enough $0<\eta\ll \e$. 
	Since $\Disc G$ is radial, for $\pi:=s/\|s\|$ we have that
	\begin{equation}\label{rad:1}
	\pi: \overline{B}^{p}_{\eta} \m \Disc G \to S_{1}^{p-1}\m \Disc G 
	\end{equation}
	is a  locally trivial fibration, and  the equality  $\pi (S_{\eta}^{p-1}\cap\Disc G) = S_{1}^{p-1}\cap \Disc G$.
	 Composing the maps \eqref{eq16} and \eqref{rad:1} one concludes that 
	\begin{equation}\label{eq17}
	\Psi_{G|}: S_{\e}^{m-1}\cap G^{-1}(\overline{B}^{p}_{\eta} \m \Disc G) \to S_{1}^{p-1} \m \Disc G
	\end{equation}	
	is a  locally trivial fibration, and its restriction to the boundary of the empty tube $S_{\e}^{m-1}\cap G^{-1}(S_{\eta}^{p-1} \m \Disc G)$	coincides with the following restriction of $G$:
	 \begin{equation}\label{eq18}
	G_{|}: S_{\e}^{m-1}\cap G^{-1}(S^{p-1}_{\eta} \m \Disc G) \to  S_{\eta}^{p-1} \m \Disc G .
	\end{equation}
	More precisely, in our case of a radial discriminant,  the bases of the fibrations \eqref{eq:emptytube} and  \eqref{eq17} can be identified  with $(\pi_{S^{p-1}_{\eta}\m \Disc G})^{-1} : S_{1}^{p-1}\m \Disc G \to S_{\eta}^{p-1}\m \Disc G$, which is  the multiplication by $\eta$. \\
	
	\noindent
	\emph{Step 2}. 	The condition \eqref{eq:reg}  is equivalent to the fact that the map $\Psi_{G} : S^{m-1}_\e \m G^{-1}(\Disc G)  \to  S_{1}^{p-1}\m \Disc G$ is a submersion (over its image) for any small enough $\e$. Consequently, the restriction
	\begin{equation}\label{eq5}
	\Psi_{G|} : S^{m-1}_\e \m \{G^{-1}(\Disc G) \cup G^{-1}(B^{p}_{\eta})\}  \to  S_{1}^{p-1}\m \Disc G  
	\end{equation}
	is a submersion. It  coincides with the fibration (\ref{eq18}) on $S_{\e}^{m-1}\cap G^{-1}(S^{p-1}_{\eta} \m \Disc (G))$. 
	Moreover, the map \eqref{eq5} is proper  since the restriction $\Psi_{G|}: S^{m-1}_{\e} \m G^{-1}(B^{p}_{\eta})\to S_{1}^{p-1}$ is a proper map, and by using \eqref{ssf}.

Finally, the fibrations (\ref{eq17}) and (\ref{eq5}) may be glued together along the fibration \eqref{eq18}  to induce  the locally trivial smooth fibration 
$$\Psi_{G|} : S_{\e}^{m-1}\m G^{-1}(\Disc G)\to S_{1}^{p-1}\m \Disc G $$
 which is independent of the small enough $\e>0$.
\end{proof}

\begin{remark}\label{rem:general}
One would like to have a more general existence result  for a sphere fibration, namely without the radiality condition, like one can prove for isolated singularities \cite{Mi}  and more generally for a point discriminant \cite[Theorem 2.1]{ACT}. 
In order to do that, we need a blow-away vector field which is tangent to $\Sing G$, 
but one does not ask anymore that the vector field is tangent to the fibres of ${\Psi_G}$ (like we ought to do here, see  \S\ref{s:equiv}).
We conjecture that this can be done for any $\Sing G$:
\end{remark}

\begin{conjecture}\label{c:conj}
 Under the conditions of Theorem \ref{sf2} but without the radiality of the discriminant,  there is a singular  open book structure on the sphere $S^{m-1}_{\e}\m G^{-1}(\Disc G)$,  independent of the small enough ${\e}$.
\end{conjecture}

 \medskip
 
 
 

  \begin{example}\label{e1}
  	Let $G: \bR^n \to \bR^2$,  $G(x_1,\ldots,x_n)=(x_1,x_{2}^{2}+\cdots+x_{n-1}^{2} - x_{n}^{2})$.  One has $V_G =\{x_1=0\} \cap \{x_{2}^{2}+\cdots+x_{n-1}^{2} - x_{n}^{2}=0\}$ and  $\Sing G = \{x_{2}=\cdots= x_{n}=0\}$, thus $\Disc G = \bR \times \{0\}$ is radial, and the map germ $G$ is \emph{nice} since it verifies \cite[Lemma 2.4]{ART} quoted in \S \ref{ss:tame}. See also Example \ref{ex4} for an alternate argument.
	
	 Since $M(G)\cap V_G \subset \Sing G \cap V_G = \{0\}$,  the condition \eqref{eq:main} is satisfied. By straightforward computations one gets that $M(\Psi)=\Sing G$, thus  $M(\Psi)\m G^{-1}(\Disc (G))=\emptyset$, thus $\Psi$ is $\rho$-regular.  By Theorem \ref{sf2} it follows that $G$ has a Milnor-Hamm sphere fibration. 
  \end{example}
  

\section{Equivalence of Milnor-Hamm tube and sphere fibrations}\label{s:equiv}

Let $G:(\bR^{m},0)\to (\bR^{p},0)$, $m>p\geq 2$,  be an analytic  nice map germ.  

\begin{problem}[The Equivalence Problem]\label{q:equiv}
Assuming that the Milnor-Hamm tube fibration and sphere fibration exist, under what conditions are they equivalent, 
in the sense that the fibrations \eqref{eq:emptytube} and \eqref{ssf} are equivalent?
\end{problem}

The equivalence problem \ref{q:equiv} makes sense in the case of a radial discriminant since the bases of the fibrations \eqref{eq:emptytube} and  \eqref{ssf} can be identified  
  via multiplication by $\eta$, as we have explained  in last part of Step 1 of the proof of Theorem \ref{sf2}.

Milnor \cite{Mi} introduced the method of ``blowing away the tube to the sphere'' which uses integration of a special vector field in order to prove the equivalence of these two fibrations in the case of non-constant holomorphic function germs $(\bC^{n}, 0) \to (\bC, 0)$.  As pointed out by Milnor \cite{Mi} and explained in all details in \cite[\S 2]{ACT}, 
this method may be applied  in the real setting under certain conditions, still referring to a point-discriminant.
We now formulate  the properties of such a  vector field in the more general context of a \emph{radial discriminant} $\Disc G$:

\begin{definition}\label{def-gvf}
One calls  \textit{Milnor vector field} for $G$ a vector field $\nu$ which satisfies the following conditions for any $x\in B^{m}_{\e} \m G^{-1}(\Disc G)$:
\begin{enumerate}
		\item [(c1)] $\nu(x)$ is tangent to the fibre  $\Psi_{G}^{-1}(\Psi_G(x))$,
		\item [(c2)] $\left\langle \nu (x), \nabla \rho(x) \right\rangle >0$, 
		\item [(c3)] $\left\langle \nu (x), \nabla \|G(x)\|^{2} \right\rangle >0$.
	\end{enumerate}
\end{definition}

We then have  the following general equivalence theorem:

\begin{theorem}\label{tm} 
	Let  $G:(\bR^{m},0)\to (\bR^{p},0)$, $m>p\geq 2$ be an analytic  nice map germ with  radial discriminant, such that  the Milnor-Hamm tube fibration \eqref{eq:tube} and sphere fibration \eqref{ssf} do exist.
	 If  there is a Milnor vector field  for $G$, then the fibrations \eqref{eq:emptytube} and \eqref{ssf} are equivalent.
\end{theorem}

 \begin{proof}
 The proof follows Milnor's pattern \cite{Mi} explained in detail in \cite{ACT}. The Milnor vector field is by definition radial pointing to the exterior of the spheres. It is well-defined and non-zero by definition outside the inverse image of the  discriminant, and the excepted set $\Disc G$ is itself radial by assumption (Definition \ref{d:radialdisc}), thus it is a collection of radii in the target. Therefore the Milnor vector field produces a flow on $B^{m}_{\e} \m G^{-1}(\Disc G)$, the projection  by $G$ of which is radial in the target. 
 
 Let us remark here that the flow does not cross the set $G^{-1}(\Disc G)$ precisely because the vector field is tangent to 
 the fibres of $\Psi_{G}$ (i.e. condition (c1)) and that the discriminant is radial.  This flow yields an 
  isotopy from the Milnor-Hamm tube fibration  to the Milnor-Hamm sphere fibration.  
  \end{proof}	

  In the real setting, the  Milnor vector field existence problem appeared first in case $\Sing G = \{0\}$ in 
 Jacquemart's work \cite{Ja1}
 and later in \cite{A1,  Oka3} etc. The more general case $\Disc  G =\{0\}$ has been addressed later  in e.g. \cite{dST0, dST1, A1, ACT, Ha, Oka4}. The authors produced sufficient conditions in each setting, but
 the  existence of a Milnor vector field without conditions seems to be an open problem\footnote{The  existence proof of  \cite[Lemma 5.2]{CSS1} has been reported as being incomplete, see  \cite[iv]{Ha} and \cite[\S 7.3]{Ri}.}.
\begin{econjecture}\label{conjec}
	Let $G:(\bR^{m},0)\to (\bR^{p},0)$, $m>p\geq 2$ be an analytic nice map germ with  radial discriminant and such that ${\Psi_G}$ is $\rho$-regular. If both fibrations \eqref{eq:tube}  and \eqref{ssf} exist, then the fibrations \eqref{eq:emptytube} and \eqref{ssf} are equivalent.
\end{econjecture}






\medskip
Let us remind that the complement   $S_{\eta}^{p-1}\m \Disc G$ may have several connected components and thus two types of fibrations over each such component. 

\begin{definition}\label{d:fibre-eq}
 We say that the fibrations \eqref{eq:emptytube} and  \eqref{ssf} are \emph{fibre-equivalent} if the corresponding fibres over each connected component of  $S_{\eta}^{p-1}\m \Disc G$ are isotopic.
\end{definition}

We show here that Conjecture \ref{conjec}
 is true at the level of fibres.

 \begin{theorem}\label{eqt1}
 	Let	$G = (G_{1}, \ldots , G_{p}): (\bR^m,0)\to (\bR^p,0)$ be an analytic nice map germ with  radial discriminant  such that $\Psi_{G}$ is $\rho$-regular. If the Milnor-Hamm fibrations \eqref{eq:tube} and \eqref{ssf} exist,  then:
	  \begin{enumerate}
  \rm \item \it   the fibrations \eqref{eq:emptytube} and \eqref{ssf} are fibre-equivalent, 
  \rm \item \it  the fibrations \eqref{eq:emptytube} and \eqref{ssf} are equivalent over any contractible component of $S^{p-1}\m \Disc G$.
   \end{enumerate}
 \end{theorem}  

\begin{proof} [Proof of Theorem  \ref{eqt1}] 
Since (b) is a simple consequence of (a), we stick to the proof of (a).
 For any vector field $\omega$ on $B_{\e}^{m} \m G^{-1}(\Disc G)$, we denote by $\proj_{T}(\omega(x))$ the orthogonal projection of $\omega(x)$ to some linear subspace $T\subset T_{x}B_{\e}^{m}= \bR^{m}$. 
  Milnor \cite{Mi} proved  the following result by using the Curve Selection Lemma, see also \cite[Lemma 1, p. 343]{Io}.
 
 \begin{lemma}\label{l:milnor} \cite{Mi}
 Let $X\subset {\bR}^m$ be an analytic manifold such that $0\in \overline{X}$. Let $f$ and $g$ be  
 analytic functions on ${\bR}^m$  such that $f(0)=g(0)=0$ and  that $f_{|X\m \{0\}} > 0$ and $g_{|X\m \{0\}} > 0$. Then there exists $\e >0$ such that for all $x \in X \cap B^{m}_{\e}$ the vectors  $\proj_{T_{x}X}(\nabla f(x))$ and  $\proj_{T_{x}X}(\nabla g(x))$ cannot have opposite direction whenever both are non-zero.
 \fin
 \end{lemma}
 
 We apply this lemma to the following situation. 
 Let $X_{y}:=\Psi_{G}^{-1}(y)$ be the fibre over some fixed value $y\in S_{1}^{p-1}\m \Disc G$. We may assume, without loss of generality, that $x$ belongs to the open set $ \{G_1(x) \neq 0\}$. By  \cite[Section 2]{dST0} and \cite[Proof of Theorem 2.2]{dST1}, the normal space $N_{x}X_{y}$ of the fibre $X_{y}$ in $\bR^{m}$  is spanned by $ \{\Omega_{2}(x),\ldots, \Omega_{p}(x)\}$ where $ \Omega_k = G_1\nabla G_k - G_k \nabla G_1, $ for $ k=2,\ldots, p$.

Consider the following vector fields on $B_{\e}^{m} \m G^{-1}(\Disc G)$: 
\[ \begin{array}{l}
v_{1}(x):=\proj_{T_{x}X_{y}}(\nabla \|G(x)\|^{2})  \\  v_{2}(x):=\proj_{T_{x}X_{y}}(\nabla \rho(x)).
\end{array}
\]

The vector field $v_{1}$ has no zeros  since the tube $\|G(x)\|^{2}= const.$ is transversal to $X_{y}$, for  $y\not\in \Disc G$.  The second vector field $v_{2}$ has no zeros on $B_{\e}^{m}\setminus G^{-1}(\Disc G)$ if and only if $M(\Psi_G)\m G^{-1}(\Disc G)=\emptyset$, i.e. $\Psi_{G}$ is $\rho$-regular, which is our assumption. 

 By Lemma \ref{l:milnor}, there is $\e >0$ such that for any $x \in X_y \cap B_\e $ either the vectors $v_{1}(x)$ and $v_{2}(x)$ are linearly independent, or they are linearly dependent but cannot have opposite direction.

Let $\e_0 >0$ such that both fibrations (\ref{eq:tube}) and (\ref{ssf}) exist, for any $0<\e < \e_0$ and $0<\eta\ll \e \leq \e_{0}$.  Let  us fix some $y \in S_{1}^{p-1}\m \Disc G$.
 
 By the above Lemma and discussion, there is  $0<\e < \e_0$  such that $ v_1(x) $ and $ v_2(x) $ cannot have strictly opposite direction  on $B_{\e}^{m}\cap X_y$. Using Milnor's original idea, we may consider the bisector vector field:  
\begin{equation}\label{eqe7}
\nu(x)= \displaystyle{\frac{v_{1}(x)}{\|v_{1}(x)\|}+\frac{v_{2}(x)}{\|v_{2}(x)\|}}
\end{equation}
 defined on $B_{\e}^{m}\setminus G^{-1}(\Disc G)$. It
 has no zeros on $X_y \cap (B_{\e}^{m} \m G^{-1}(\Disc G))$ precisely because $v_{1}$ and $v_{2}$ do not point in opposite directions. This is thus a Milnor vector field on $B_{\e}^{m}\cap X_y$ and we may now apply Theorem \ref{tm} to the
 restriction of $G$ to the space $B_{\e}^{m}\cap X_y$ in order to prove the isotopy of the fibres over $y$ of the two fibrations. 
 
 Varying the point $y\in S_{1}^{p-1}\m \Disc G$ we get the isotopy of the corresponding fibres over any connected component.
 \end{proof}


\subsection{The importance of the Milnor set $M(G)$}
 
We  tacitly assume throughout this subsection that $G$ is an  analytic nice map with radial discriminant and that the Milnor-Hamm fibrations \eqref{eq:tube} and \eqref{ssf} exist.
Let us show that the obstruction to the existence of a Milnor vector field  is the Milnor set.
\begin{lemma}\label{eql1}
Let  $x\in B_{\e}\m G^{-1}(\Disc G)$. 
 The vectors $v_{1}$ and $v_{2}$  are linearly dependent  if and only if  $x \in M(G)$.
 In particular the vector field  $\nu$ from \eqref{eqe7} is a Milnor vector field on $B_{\e}^m \m (M(G)\cup G^{-1}(\Disc G))$. 
\end{lemma}
\begin{proof}
For $x\in B_{\e} \setminus G^{-1}(\Disc G)$ we consider the decomposition:
	\begin{equation}\label{eqe8}
	\displaystyle{T_{x} X_{y} =T_{x}G^{-1}(G(x))\oplus \bR\langle v_1(x)\rangle},
	\end{equation}
	where, by its definition, the vector $v_1(x)$ is orthogonal to  $T_{x}G^{-1}(G(x))$ in $T_{x} X_{y}$. 
			
One writes $v_{1}(x)=v_{1}^{1}(x)+v_{1}^{2}(x)$ and  $v_{2}(x)=v_{2}^{1}(x)+v_{2}^{2}(x)$ according to the decomposition \eqref{eqe8}.
 From the definitions, we have  $v^{1}_{1}(x) =0$  for all $x\in B_{\e}\m G^{-1}(\Disc G)$, and $v^{1}_{2}(x)= 0$ if and only if  $x  \in M(G)$.  This proves the first claim, and the second is an easy consequence.
\end{proof}

For $x\in M(G)\m G^{-1}(\Disc G)$ one has that $v_{1}(x)$ and $v_{2}(x)$ are collinear, which amounts to the relation:
\begin{equation}\label{rho}
\nabla \rho (x) =a(x) \nabla \|G(x)\|^2 + \sum_{j=2}^{p}b_j(x)\Omega_j (x),
\end{equation}
where
	\begin{equation}\label{coef}
	a(x)= \frac{\langle \nabla \rho (x),v_{1}(x)\rangle}{\|v_{1}(x) \|^2}.
	\end{equation}

This proves in particular that  $x\in M(\Psi_G)\m G^{-1}(\Disc G)$ $\Leftrightarrow$ $a(x)=0$.

With these notations one may characterise the existence of a Milnor vector field for $G$ as follows,  extending the particular case $\Disc G=\{0\}$  of \cite[Theorem 3.3.1]{Ha}.

\begin{theorem}\label{eqt2} \
Let $G:(\bR^{m},0)\to (\bR^{p},0),$ $m\geq p \geq 2$ be an analytic nice map with radial discriminant
such that the Milnor-Hamm fibrations \eqref{eq:tube} and \eqref{ssf} exist.
There exists a Milnor vector field for $G$ on $B_{\e}\m G^{-1}(\Disc G),$ for some small enough $\e >0$ if and only if
 $a(x)>0$ for any $x\in M(G)\m G^{-1}(\Disc G)$. 
\end{theorem}

\begin{proof}
The implication ``$\Rightarrow$'' follows from the definition \eqref{eqe7} of the vector field $\nu(x)$. In fact, by condition (c1) one has that $\left\langle \nu(x), \Omega_j(x) \right\rangle=0 $ for any $j=2,\ldots,p$. Therefore, $\left\langle \nabla \rho (x),\nu(x) \right\rangle = a(x)\left\langle \nabla \|G(x)\|^2, \nu(x)\right\rangle$, which  by (c2) and (c3) implies that $a(x)>0$. 
	
Reciprocally,   if $a(x)>0$, it follows from Lemma \ref{eql1} and  \eqref{rho}
 that the  vector field $\nu(x)$ has no zeroes on $B^{m}_{\e} \m V_G$, hence it is a Milnor vector field for $G$.
\end{proof}

There are several other criteria for the existence of a Milnor vector field for $G$; we discuss some of them in \cite{AR}.



\begin{proposition}\label{p:milnorset}
  The image of the restriction $G_{|}: B^{m}_{\e}\cap M(G)\m G^{-1}(\Disc G) \to \bR^{p}$  contains $B^{p}_{\eta}\cap \im G\m \Disc G$
 for some small enough ball $B^{p}_{\eta}$ centred at  the origin.
\end{proposition}
\begin{proof}
Let $y\in B^{p}_{\eta}\cap\im G \m \Disc G$, for some $\eta >0$ which fits in the Milnor-Hamm tube fibration, and such that the fibre $G^{-1}(y)$ is not empty. 
The distance function to the origin $\rho_{|}: \overline{B^{m}_{\e}}\cap \overline{G^{-1}(y)} \to \bR_{\geq 0}$ has at least one local minimum point $x_y$ in the interior  $B^{m}_{\e}\cap G^{-1}(y)$. This implies that   
$x_{y} \in M(G)$. 
\end{proof}
Let $\bigcup_{\beta}M_{\beta}$ be the decomposition into connected components of the subanalytic set germ at the origin $M(G)\m G^{-1}(\Disc G)$.  See Example \ref{e:non-connected} where we have 8  connected components.
 \begin{corollary}\label{c:milnorset}
 If one of the following conditions holds:
\begin{enumerate}
  \rm \item \it  $M(G)\m G^{-1}(\Disc G)$ is connected, 
  \rm \item \it for any $\beta$ there is $y= y_{\beta}\in S_{1}^{p-1}\m \Disc G$ such that  the germ at 0 of $M_{\beta} \cap X_{y}$ has dimension $>0$,
\end{enumerate}
 then there exists a Milnor vector field for $G$. 
\end{corollary}
\begin{proof}
(a).  Proposition \ref{p:milnorset} implies  that  $\dim M(G) \cap X_{y} >0$ for any $y\in S_{1}^{p-1}\m \Disc G$, thus (a) 
 is a particular case of (b).
 
 \noindent
 (b). For some fixed $\beta$ we apply Milnor's Lemma \ref{l:milnor} to that $X_{y}$ for which $\dim M_{\beta} \cap X_{y} >0$. This shows that the vectors $v_{1}(x)$ and $v_{2}(x)$ point in the same direction for all $x\in M_{\beta}\cap X_{y}$ and hence for all $x\in M_{\beta}$ since this is connected.  Since this is true for any $\beta$ it is then true for the whole set $M(G)\m G^{-1}(\Disc G)$, thus  the vector field $\nu$ has no zeroes on it.
Finally we may apply Lemma \ref{eql1} to conclude that $\nu$ is a Milnor vector field.
\end{proof}

 \section{Classes of maps with equivalent Milnor-Hamm fibrations}

 
 \subsection{Mixed  functions and the equivalence problem}
 
 \begin{definition}\label{d:polar} 
 The mixed function $F:\bC^{n} \to \bC$ is called \emph{polar} weighted-homogeneous of degree $k$ if there are non-zero integers 
$p_{1},\ldots,p_{n}$ and $k>0$, such that $\gcd(p_{1},\ldots p_{n})=1$ and $\sum_{j=1}^{n}p_{j}(\nu_{j}-\mu_{j})=k$, for any monomial of the expansion $F(\mathbf{z}) = F(\mathbf{z},\mathbf{\bar{z}})=\underset{{\scriptstyle \nu,\mu}}{\sum}c_{\nu,\mu}\mathbf{z}^{\nu}\mathbf{\bar{z}^{\mu}}$.

The corresponding  $S^{1}$-action on $\mathbb{C}^{n}$ is, for $\lambda\in S^{1}$: 
\[ 
\lambda\cdot(\mathbf{z},\mathbf{\overline{z}})  =(\lambda^{p_{1}}z_{1},\ldots,\lambda^{p_{n}}z_{n},\lambda^{-p_{1}}\overline{z}_{1},\ldots,\lambda^{-p_{n}}\overline{z}_{n}).
\] 
\end{definition}

\smallskip

 \begin{theorem}\label{t:polar}
 	Let $F:\bC^n \to \bC$ be a polar weighted homogeneous mixed function. Then $F$ is a nice map germ with 
	$\Disc F = \{ 0\}$,  the Milnor fibrations \eqref{eq:tube} and \eqref{ssf} exist, and  the fibrations \eqref{eq:emptytube} and  \eqref{ssf} are equivalent.	
 \end{theorem}
\begin{proof}
From \cite[\S 4.1]{ACT} it follows (due to the $S^{1}$-action) that  polar weighted-homogeneous maps  are \emph{nice}, more precisely because the image $F(B_{\e}^{2n})$ contains a small neighbourhood of the origin for any $\e>0$, and necessarily $\Disc F = \{ 0\}$. The discriminant is thus trivially radial. 

It was proved in  \cite[Theorem 5.2]{PT} that a polar weighted-homogeneous mixed function germ $F$ has  tube fibration,  and in \cite[Theorem 1.4]{ACT} that such $F$ has a sphere fibration. It therefore remains to show that they are equivalent.

Each  component $M_\beta$ of the decomposition into connected components $\bigcup_{\beta}M_{\beta}$ of the semi-analytic set germ at the origin $M(F) \m V_F$ is invariant under the $S^{1}$-action, more precisely one has that $ \lambda\, M_\beta = M_\beta$
for any  $\lambda \in S^{1}$.  It follows that this verifies the hypothesis (b) of Corollary \ref{c:milnorset} and thus our claim follows.
\end{proof}

 
 \subsection{Maps with radial action}
 Let $t\cdot x := (t^{q_1}x_1, \ldots  , t^{q_m} x_m)$
 for $ t\in \bR_{+}$ and $q_1, \ldots, q_m \in \mathbb{N}^{\ast}$ relatively prime positive integers. One says that the map
  $G=(G_{1}, \cdots, G_{p}) :{\bR}^{m} \to {\bR}^{p}$ is  \textit{radial weighted-homogeneous} (or radial, for short) of weights $(q_{1}, \ldots, q_{m}) $ and of degree $d >0,$ if   $ G(t\cdot x)=t^d G(x)$. 
  
  \smallskip
  
 \begin{theorem}\label{t2}
 	Let $G:(\bR^m,0) \to (\bR^p,0)$ be a radial weighted homogeneous map germ, and satisfying the condition \eqref{eq:main}. Then $G$ is nice, with radial discriminant,  has Milnor-Hamm tube and sphere fibrations, and the fibrations \eqref{eq:emptytube} and \eqref{ssf} are equivalent.
 \end{theorem}

\begin{proof}
 	The image of $G$ is a real cone and this cone is stable as a germ, in the sense that $G(B^{m}_{\e})$ and $G(B^{m}_{\e'})$ have the same germs at the origin, for any $0< \e' < \e$.  Moreover, the boundary $\partial\overline{\im G}$ is also a conical set germ at the origin.
	
	The image by G of any analytic germ $X\subset \bR^{m}$
	which is invariant under the $\bR_{+}$-action is a conical germ, and  $\Sing G$ is such an invariant set germ. 
	It follows that 
  $\Disc G$ is well-defined as a germ, and it is radial.  These show that $G$ is a nice map germ, without using the hypothesis
  about condition \eqref{eq:main}.
  
 The assumed condition \eqref{eq:main} insures now the existence of the Milnor-Hamm tube fibration via \cite[Lemma 3.3]{ART}.  Let us see that the Milnor-Hamm sphere fibration exists too.  It was proved in \cite[Prop. 3.2]{ACT} by using the Euler vector field 
$\gamma(x):=\sum_{j=1}^{m}q_{j}x_{j}({\partial}/{\partial x_{j}})$  that the spheres are transversal to the fibres of the map $\Psi_{G}$. In our setting this implies that $\Psi_{G}$ is $\rho$-regular, thus our claim follows by  Theorem \ref{sf2}.
  
   The existence of a Milnor vector field follows by noting that any connected component $M_\beta$ of $M(F) \m V_F$ is also invariant under the $\bR_{+}$-action, and thus we may  apply Corollary \ref{c:milnorset}(b).
\end{proof}

  \smallskip
  
  \begin{example}\cite[Example 5.6]{ART}\label{e:non-connected}
Let $G:(\bR^3,0)\to (\bR^2,0)$ given by $G(x,y,z)=(xy,z^2)$ is radial homogeneous. One has $V_G =\{x=z=0\}\cup \{y=z=0\}$, $\Sing G=\{z=0\}\cup \{x=y=0\}$, $\Disc G= \{(0,\beta)\,|\,\beta \ge0\}\cup \{(\lambda,0)\,|\, \lambda \in \bR\}$, and $G^{-1}(\Disc G)=\{x=0\}\cup\{y=0\}\cup \{z=0\}$.
We see that $\Disc G$ is radial, as predicted by Theorem \ref{t2}. 

By further computations one gets $M(G)= \{x=\pm y\}\cup\{z=0\}$. To check that $G$ satisfies the condition \eqref{eq:main}, let us  consider $p_0=(x_0,y_0,z_0)\in \overline{M(G)\m G^{-1}(\Disc G)}\cap V_G$.  Then there is a sequence $p_n:=(x_n,y_n,z_n)\in M(G)\m G^{-1}(\Disc G)$ such that $p_n \to p_0$ with $p_0\in V_G$. Consequently, $z_0 = 0$ and $x_n=\pm y_n \neq 0$ since $p_{n}\not \in G^{-1}(\Disc G)$.
Thus $x_0=\lim x_n=\pm \lim y_n=y_0=0,$ and therefore $p_0=(0,0,0)$.	

Then by Theorem \ref{t2} the map germ $G$ has Milnor-Hamm tube and sphere fibration, and the fibrations \eqref{eq:emptytube} and  \eqref{ssf} are equivalent.
 \end{example}

 \begin{corollary}\label{t3}
 	Let $(f,g)$ be a holomorphic map germ which is Thom regular  at $V_{(f,g)}$, and such that $f$ and $g$ do not have common factor of order $>0$.
	
	 If $f\bar g$ is a radial weighted homogeneous function, then $f\bar g$ has  Milnor-Hamm  tube and sphere fibrations,  and the fibrations \eqref{eq:emptytube} and \eqref{ssf} are equivalent.
 \end{corollary}
 \begin{proof}
The Thom regularity of $(f,g)$ implies the Thom regularity of $f\bar g$ by \cite[Theorem 4.3]{ART}  which extends \cite[Theorem 3.1]{PT}, and thus condition \eqref{eq:main} is verified and we may apply the above Theorem  \ref{t2} to conclude.
 \end{proof}
 
   \smallskip
 
 \begin{corollary}\label{spr2}
 	Let $f$ and $g$ be holomorphic, radial weighted-homogeneous such that the map germ $(f,g)$ is an ICIS. Then the map germ $f\bar g: (\bC^n, 0) \to (\bC, 0)$ has  Milnor-Hamm tube and sphere fibrations,  and  the fibrations \eqref{eq:emptytube} and \eqref{ssf} are equivalent.
 \end{corollary}
 
 \begin{proof}
Since  $(f,g)$ is an ICIS, it follows by \cite[Theorem 2.7]{ART} that the map germ $f\bar{g}$ is a nice map germ,  and that it is Thom regular and has a Milnor-Hamm tube fibration, by \cite[Theorem 4.3(a)]{ART}. If we add up the $\bR_{+}$-action then we get, as in Theorem  \ref{t2}  above, the existence of 
a Milnor-Hamm sphere fibration and thus the equivalence of the fibrations.
 \end{proof}
 
 Let us point out that if $(f,g)$ is an ICIS  then its discriminant  $\Disc f\bar g$ is either a union of semi-analytic curves or it is $\{0\}$. This is actually true for any complex analytic map germ $(f,g)$, as shown in \cite[Theorem 2.3]{PT2} and in  \cite{Oka5}, together with  precise criteria for $\Disc f\bar g$ being positive dimensional.  
  
 \begin{example} 
 Let us consider the map $f\bar g$, where $f,g :(\mathbb{C}^2, 0) \to (\mathbb{C},0)$,  $f(x,y) =x^{2}+y^{2}$ and $g(x,y)=  x^{2}-y^{2}$. One easily computes that $\Disc f\bar{g}$ is the real axis in $\bC$ (see e.g. \cite{PT,PT2}).
This verifies the assumptions of Corollary \ref{spr2}, thus  $f\bar g$ has Milnor-Hamm tube and sphere fibrations, and they are equivalent.
 \end{example}


\end{document}